\documentclass{elsarticle}
\usepackage{amsmath}
\usepackage{amsfonts}
\usepackage{amsthm}
\usepackage{comment}

\DeclareMathOperator{\Hom}{Hom}
\DeclareMathOperator{\End}{End}
\DeclareMathOperator{\GL}{GL}

\theoremstyle{plain}
\newtheorem{theorem}{Theorem}[section]
\newtheorem{proposition}[theorem]{Proposition}
\newtheorem{lemma}[theorem]{Lemma}
\newtheorem{corollary}[theorem]{Corollary}

\theoremstyle{definition}
\newtheorem{definition}[theorem]{Definition}
\newtheorem{remark}[theorem]{Remark}
\newtheorem{example}[theorem]{Example}

\title{Computing isomorphism numbers of \\ $F$-crystals using the level torsions}
\author{Xiao Xiao\corref{cor1}}
\address{Mathematics Department, Utica College, 1600 Burrstone Road, Utica, NY, 13502}
\cortext[cor1]{Email: xixiao@utica.edu}
\date{}

\begin{document}
\begin{abstract}
The isomorphism number of an $F$-crystal $(M, \varphi)$ over an algebraically closed field of positive characteristic is the smallest non-negative integer $n_M$ such that the $n_M$-th level truncation of $(M, \varphi)$ determines the isomorphism class of $(M, \varphi)$. When $(M, \varphi)$ is isoclinic, namely it has a unique Newton slope $\lambda$, we provide an efficiently computable upper bound for $n_M$ in terms of $\lambda$ and the Hodge slopes of $(M, \varphi)$. This is achieved by providing an upper bound for the level torsion of $(M, \varphi)$ introduced by Vasiu. We also check that this upper bound is optimal for many families of isoclinic $F$-crystals that are of special interest (such as isoclinic $F$-crystals of K3 type).  
\end{abstract}

\maketitle 

\section{Introduction}

Let $p$ be a prime number and $k$ an algebraically closed field of characteristic $p$. It has been known for many years that the isomorphism class of a $p$-divisible group $D$ over $k$ is determined by a finite truncation $D[p^n]$ of $D$. The smallest integer $n$ with the property that $D[p^n]$ determines $D$ is called the \emph{isomorphism number} of $D$ and denoted by $n_D$. Only recently, Lau, Nicole and Vasiu \cite{Vasiu:traversosolved} discovered an optimal upper bound for this number in terms of the Hodge polygon and the Newton polygon of $D$. The isomorphism number of an $F$-crystal is the generalization of the isomorphism number of a $p$-divisible group (see Definition \ref{def:isonum} for the precise definition). In this paper we provide an upper bound for the isomorphism number of an arbitrary isoclinic $F$-crystal (i.e. those having a unique Newton slope) in terms of its Hodge polygon and Newton polygon. It not only recovers the optimal upper bound in the isoclinic $p$-divisible groups case, but also provides optimal upper bounds in various other cases. Let us describe our results.

We fix the prime number $p$ and the ground field $k=\bar{k}$ throughout this paper. Let $W = W(k)$ be the ring of Witt vectors over $k$ and $K_0$ its field of fractions. Let $\sigma$ be the Frobenius automorphism of $W$ and $K_0$. An \emph{$F$-crystal} over $k$ is a pair $(M, \varphi)$ where $M$ is a free $W$-module of finite rank $r$ and $\varphi$ is a $\sigma$-linear injective endomorphism of $M$. If $pM \subset \varphi(M)$, then the $F$-crystal $(M, \varphi)$ is called a \emph{Dieudonn\'e module} over $k$. For the rest of this paper, all $F$-crystals, Dieudonn\'e modules and $p$-divisible groups are over $k$ unless otherwise stated. For $n = 1, 2, \dots$, let $W_n :=W/(p^n)$ be the ring of Witt vectors of length $n$ with coefficients in $k$. The \emph{Hodge slopes} $e_1, e_2, \dots, e_r$ of an $F$-crystal $(M, \varphi)$ are the non-negative integers such that $M / \varphi(M) \cong \bigoplus_{i=1}^r W_{e_i}$ as $W$-modules. By reindexing, we can assume that $e_1 \leq e_2 \leq \cdots \leq e_r$. The \emph{$i$-th Hodge number} of $(M, \varphi)$ is $h_i := \#\{j \; | \; e_j = i\}$. Dieudonn\'e \cite[Theorems 1, 2]{Dieudonne:fcrystalclassify} and Manin \cite[Chapter 2, Section 4]{Manin:formalgroups}'s classification of $F$-isocrystals implies that there is a direct sum decomposition $(M \otimes_W K_0, \varphi) \cong \bigoplus_{\lambda \in \mathbb{Q}_{\geq 0}} E^{m_{\lambda}}_{\lambda}$, where each $E_{\lambda}$ is the simple $F$-isocrystal with all Newton slopes equal to $\lambda$ and the multiplicity $m_{\lambda} \in \mathbb{Z}_{\geq 0}$ is uniquely determined and is zero for all but finitely many $\lambda$. An $F$-crystal $(M, \varphi)$ is called \emph{isoclinic} if $(M \otimes_W K_0, \varphi)$ is isomorphic to $E^{m_{\lambda}}_{\lambda}$ for some $\lambda \in \mathbb{Q}_{\geq 0}$. Let $\GL(M)$ be the group of $W$-linear automorphism of $M$.

\begin{definition} \label{def:isonum}
The \emph{isomorphism number} $n_M$ of an $F$-crystal $(M, \varphi)$ over $k$ is the smallest non-negative integer such that for every $g \in \GL(M)$ with the property that $g \equiv 1$ mod $p^{n_M}$, the $F$-crystal $(M, g\varphi)$ is isomorphic to $(M, \varphi)$. 
\end{definition}

By classical Dieudonn\'e theory, the category of $p$-divisible groups over $k$ is anti-equivalent to the category of Dieudonn\'e modules over $k$; see \cite[Chapter 3]{Demazure1}. Under this correspondence, the isomorphism number of a $p$-divisible group is equal to the isomorphism number of the corresponding Dieudonn\'e module; see \cite[Corollary 3.2.2]{Vasiu:CBP}. On the other hand, the isomorphism number $n_M$ of an $F$-crystal $(M, \varphi)$ is the smallest non-negative integer such that the $F$-truncation mod $p^{n_M}$ of $(M, \varphi)$ determines the isomorphism class of $(M, \varphi)$; see \cite[Section 3.2.9]{Vasiu:CBP} for the definition of $F$-truncation and \cite[Section 3.3]{Xiao:thesis} for the proof. The last two sentences imply that Definition \ref{def:isonum} is the right definition for the isomorphism numbers of $F$-crystals which generalizes the isomorphism numbers of $p$-divisible groups. Early works of Manin \cite[Theorems 3.4 and 3.5]{Manin:formalgroups} imply that $n_D$ exists for any Dieudonn\'e module. Recently, Vasiu showed that $n_M$ exists in general; see \cite[Main Theorem A]{Vasiu:CBP}.

Let $c$ and $d$ be the codimension and dimension of a $p$-divisible group $D$ respectively. Traverso proved that $n_D \leq cd+1$ \cite[Theorem 3]{Traverso:pisa} and later conjectured that $n_D \leq \min\{c, d\}$ \cite[Section 40, Conjecture 4]{Traverso:specializations}. Since then, the conjecture has been verified in various cases, for example, in the cases of supersingular $p$-divisible groups \cite[Theorem 1.2]{Vasiu:supersingular} and quasi-special $p$-divisible groups \cite[Theorem 1.5.2]{Vasiu:reconstructing}. Only recently, Lau, Nicole and Vasiu \cite[Theorem 1.4]{Vasiu:traversosolved} found an optimal upper bound $n_D \leq \lfloor 2cd/(c+d) \rfloor$ which proves a corrected version of Traverso's conjecture. In the search for optimal upper bounds for $n_D$, the following play important roles:
\begin{itemize}
\item Classical Dieudonn\'e theory of $p$-divisible groups over $k$. This allows us to use tools on the geometric side as well as the algebraic side.
\item Deformation theory of $p$-divisible groups over general schemes. Let $(M, \varphi)$ be a Dieudonn\'e module over $k$. One useful result in deformation theory allows us to assume that the dimension of $M/(\varphi(M) + \varphi^{-1}(pM))$ as a $k$-vector space is $1$; see \cite[Proposition 2.8]{Oort:moduli}. With this assumption, every Dieudonn\'e module over $k$ has a $W$-basis that is well-suited to computations. 
\item The study of the level torsion (see Subsection \ref{section:level torsion} for the definition) of Dieudonn\'e modules \cite{Vasiu:reconstructing}. The main result of loc. cit. provides a computable upper bound for the isomorphism numbers; see Theorem \ref{theorem:vasiu1}.
\end{itemize}

Unfortunately, to find optimal upper bounds for $n_M$ for more general $F$-crystals, we do not have as many tools as we have in the case of $p$-divisible groups. For instance, there is no general way to deform $F$-crystals. However, the level torsion of an $F$-crystal is well-defined and has been studied in \cite{Vasiu:reconstructing}. In this paper we will use the level torsion to provide a good upper bound for the isomorphism number of isoclinic $F$-crystals. 

\begin{theorem} \label{theorem:estimate_n_elementary}
Let $(M, \varphi)$ be an isoclinic $F$-crystal over $k$ with Hodge numbers $h_1, h_2, \dots$ and unique Newton slope $\lambda$. If the smallest and the largest Hodge slopes of $(M, \varphi)$ are $0$ and $e$ respectively,  then the isomorphism number $n_M$ of $(M, \varphi)$ satisfies the following inequality:
\begin{equation} \label{equation:main}
n_M \leq \lfloor e \sum_{i>\lambda}h_i + (\sum_{i<\lambda}h_i-\sum_{i>\lambda}h_i)\lambda \rfloor.
\end{equation}
\end{theorem}

By Remark \ref{remark:rescale}, every $F$-crystal can be rescaled so that its smallest Hodge slope is equal to zero without changing its isomorphism number, thus the assumption that the smallest Hodge slope is equal to zero in Theorem \ref{theorem:estimate_n_elementary} is not restrictive. We mention that, even though Theorem \ref{theorem:estimate_n_elementary} recovers the optimal upper bound in the isoclinic $p$-divisible groups case as found by Lau, Nicole and Vasiu (see Corollary \ref{corollary:recover}), it does not assert that the upper bound is indeed optimal. It is possible to improve Theorem \ref{theorem:estimate_n_elementary} in some cases; see Example \ref{example:est_n_ele_not_optimal}. By using Theorem \ref{theorem:estimate_n_elementary}, we can compute optimal upper bounds for the isomorphism numbers in a few special cases, as we now describe.

An $F$-crystal of rank $r$ is called of \emph{K3 type} if its Hodge numbers are $h_0 = 1$, $h_1 = r-2$, $h_2 = 1$ and $h_i = 0$ for all $i \geq 2$. An $F$-crystal of K3 type with $r=21$ relates to the second crystalline cohomology group of K3 surfaces over $k$, thanks to a theorem of Mazur \cite[Theorem 2]{Mazur:FrobeniusHodge}. 

\begin{theorem} \label{theorem:K3}
Let $(M, \varphi)$ be a direct sum of $F$-crystals of K3 type. Then $n_M \leq 2$. Moreover, 
\begin{enumerate}[(i)]
\item if $(M, \varphi)$ is a direct sum of non-isoclinic $F$-crystals of K3 type, then $n_M = 1$; 
\item if $(M, \varphi)$ is a direct sum of isoclinic $F$-crystals of K3 type, then $n_M = 2$;
\item if $(M, \varphi)$ is a mixed direct sum of non-isoclinic and isoclinic $F$-crystals of K3 type, then $n_M = 2$.
\end{enumerate}
\end{theorem}
We prove Theorem \ref{theorem:K3} in Section \ref{section:K3}. Its proof uses Theorem \ref{theorem:estimate_n_elementary} in the isoclinic case and the Newton-Hodge decomposition theorem \cite[Theorem 1.6.1]{Katz:slopefiltration} in the non-isoclinic case.

\begin{theorem} \label{theorem:rank2}
Let $(M, \varphi)$ be an $F$-crystal of rank $2$ with Hodge slopes $0$ and $e>0$. Let $\lambda_1$ be the smallest Newton slope of $(M, \varphi)$. Then we have
\begin{enumerate}[(i)]
\item if $(M, \varphi)$ is a direct sum of two $F$-crystals of rank $1$, then $n_M = 1$;
\item if $(M, \varphi)$ is not a direct sum of two $F$-crystals of rank $1$ and is isoclinic, then $n_M = e$;
\item if $(M, \varphi)$ is not a direct sum of two $F$-crystals of rank $1$ and is non-isoclinic, then $n_M \leq 2\lambda_1$.
\end{enumerate}
\end{theorem}
We prove Theorem \ref{theorem:rank2} in Section \ref{section:rank2}. Part (i) is an easy consequence of Corollary \ref{corollary:directsumordinary}.  We use Theorem \ref{theorem:estimate_n_elementary} to prove part (ii). For part (iii), as the rank $2$ is small, we estimate the level torsion by brute force and thus get an upper bound for the isomorphism number. 

Following \cite[Definition 1.5.1]{Vasiu:reconstructing}, we make the following definitions.
\begin{definition} \label{definition:quasi-special}
An $F$-crystal $(M, \varphi)$ of rank $r$ is called an \emph{isoclinic quasi-special $F$-crystal} if $\varphi^r(M) = p^sM$ for some integer $s$. If $(M, \varphi)$ is a direct sum of isoclinic quasi-special $F$-crystals, then it is called a \emph{quasi-special $F$-crystal}.
\end{definition}
In fact, the integer $s$ must be the sum of all Hodge slopes of $(M, \varphi)$; see Lemma \ref{lemma:smust}. Quasi-special $F$-crystals are the generalization of quasi-special Dieudonn\'e modules \cite[Definition 1.5.1]{Vasiu:reconstructing}. Moreover, they generalize special Dieudonn\'e modules \cite[Definition 3.2.3]{Manin:formalgroups}.

\begin{theorem} \label{theorem:quasi-special}
Let $(M, \varphi)$ be a quasi-special $F$-crystal. Suppose $(M, \varphi)$ has Hodge slopes $e_1 \leq e_2 \leq \dots \leq e_r$ and set $s := \sum_{i=1}^r e_i$. The following inequality holds:
\[n_M \leq \min\{s, re_r-s\}.\]
\end{theorem}

We note that Theorem \ref{theorem:quasi-special} is not always optimal; see Example \ref{example:quasispecialoptimal} and Remark \ref{remark:quasispecial}.
\vskip 0.15in
\noindent \textbf{Acknowledgements}. The author would like to thank Adrian Vasiu for suggesting this problem, and for numerous valuable discussions and comments on this and related topics over the course of this project. The author would also like to thank the referee for many valuable comments and suggestions which in particular led to a shorter way to prove Lemma \ref{lemma:n=0}, and a much better way to present the proof of Theorem \ref{theorem:estimate_n_elementary} via Lemma \ref{lemma:mainlemma2}.

\section{Preliminaries} \label{section:pre}
\subsection{Notations}
A \emph{latticed $F$-isocrystal} over $k$ is a pair $(M, \varphi)$, where $M$ is a free $W$-module of finite rank $r$ and $\varphi$ is a $\sigma$-linear automorphism of $M \otimes_{W} K_0$. For the sake of simplicity, we denote $M \otimes_{W} K_0$ by $M[1/p]$ for the rest of this paper. Recall that if $\varphi(M) \subset M$, then $(M, \varphi)$ is called an $F$-crystal over $k$. Moreover, if $pM \subset \varphi(M)$, then $(M, \varphi)$ is called a Dieudonn\'e module over $k$. 

Let $(M_1, \varphi_1)$ and $(M_2, \varphi_2)$ be two latticed $F$-isocrystals. The set of all $W$-linear homomorphisms from $M_1$ to $M_2$ is a free $W$-module, denoted by $\Hom(M_1,M_2)$. Let $\varphi_{12}$ be the $\sigma$-linear automorphism of $\Hom(M_1[1/p], M_2[1/p])$ defined by the following rule: for any $f \in \Hom(M_1[1/p], M_2[1/p])$, let 
\[\varphi_{12}(f) := \varphi_2 \circ f \circ \varphi_1^{-1} \in \Hom(M_1[1/p], M_2[1/p]).\]
As $\Hom(M_1, M_2)[1/p] \cong \Hom(M_1[1/p], M_2[1/p])$ as $K_0$-vector spaces, the pair $(\Hom(M_1, M_2), \varphi_{12})$ is a latticed $F$-isocrystal. If $(M_1, \varphi_1)=(M_2, \varphi_2)$, then $(\Hom(M_1, M_2), \varphi_{12})$ is denoted by $(\End(M_1), \varphi_1)$. If $(M_2, \varphi_2) = (W, \sigma)$, then $(\Hom(M_1, M_2), \varphi_{12})$ is denoted by $(M_1^*, \varphi_1)$ and called the \emph{dual} of $(M_1, \varphi_1)$. The isomorphism number of a latticed $F$-isocrystal can be defined in the same way as the isomorphism number of an $F$-crystal. Moreover, the isomorphism number of a latticed $F$-isocrystal is invariant under duality. See \cite[Fact 4.2.1]{Vasiu:reconstructing} for a proof in the Dieudonn\'e module case, which is easily adapted to the latticed $F$-isocrystal case.

\begin{lemma} \label{lemma:rescale}
Let $(M, \varphi)$ be a latticed $F$-isocrystal over $k$ and let $n_M$ be its isomorphism number. For all $m \in \mathbb{Z}$, the isomorphism number of the latticed $F$-isocrystal $(M, p^m\varphi)$ is also $n_M$.
\end{lemma}
\begin{proof}
The proof is straightforward. For details, see \cite[Proposition 3.4]{Xiao:thesis}.
\end{proof}

\begin{remark} \label{remark:rescale}
By Lemma \ref{lemma:rescale}, we can assume that the smallest Hodge slope of an $F$-crystal $(M, \varphi)$ is zero without changing its isomorphism number by multiplying an appropriate power of $p$ to $\varphi$.
\end{remark}

\subsection{The level torsion} \label{section:level torsion}
We now recall the definition of the level torsion from \cite{Vasiu:reconstructing}. It is the main tool to find good upper bounds for the isomorphism number of latticed $F$-isocrystals.

Let $(M, \varphi)$ be a latticed $F$-isocrystal. By using Dieudonn\'e \cite[Theorems 1, 2]{Dieudonne:fcrystalclassify} and Manin's \cite[Chapter 2, Section 4]{Manin:formalgroups} classification of $F$-isocrystals, we obtain a direct sum decomposition \[\End(M[1/p]) \cong L_+ \oplus L_0 \oplus L_-\] into $K_0$-vector spaces, where
\[L_+ = \bigoplus_{\lambda_1 < \lambda_2} \Hom(E_{\lambda_1}, E_{\lambda_2}), \quad L_0 = \bigoplus_{\lambda} \End(E_{\lambda}), \quad L_- = \bigoplus_{\lambda_1 < \lambda_2} \Hom(E_{\lambda_2}, E_{\lambda_1}).\]
Define
\[O_+ = \bigcap_{i=0}^{\infty} \varphi^{-i}(\End(M) \cap L_+), \quad O_- = \bigcap_{i=0}^{\infty} \varphi^i(\End(M) \cap L_-), 
\]
\[O_0 = \bigcap_{i=0}^{\infty} \varphi^{-i}(\End(M) \cap L_0) = \bigcap_{i=0}^{\infty} \varphi^i(\End(M) \cap L_0). \]
For $* \in \{+, 0, -\}$, each $O_*$ is a lattice of $L_*$. We have the following relations:
\[\varphi(O_+) \subset O_+, \quad \varphi(O_0) = O_0 = \varphi^{-1}(O_0), \quad \varphi^{-1}(O_-) \subset O_-.\]
Write $O := O_+ \oplus O_0 \oplus O_-$; it is a lattice of $\End(M)[1/p]$ sitting inside $\End(M)$. The \emph{level torsion} $\ell_M$ of $(M, \varphi)$ is defined by the following two disjoint rules:
\begin{enumerate}[(i)]
\item if $O = \End(M)$ and the ideal generated by $O_+ \oplus O_-$ is not topologically nilpotent, then the level torsion $\ell_M :=1$;
\item in all other cases, the level torsion $\ell_M$ is the smallest non-negative integer such that $p^{\ell_M} \End(M) \subset O$.
\end{enumerate}
 
Vasiu proved the following important theorem:

\begin{theorem} \label{theorem:vasiu1}
Let $(M,\varphi)$ be a latticed $F$-isocrystal. Then $n_M \leq \ell_M$. Moreover, if $(M,\varphi)$ is a direct sum of isoclinic latticed $F$-isocrystals, then $n_M = \ell_M$.
\end{theorem}
\begin{proof}
See \cite[Main Theorem A]{Vasiu:reconstructing}.
\end{proof}

\subsubsection{Computing the level torsion of isoclinic $F$-crystals} 

Let $(M, \varphi)$ be a latticed $F$-isocrystal. Following \cite[Definitions 4.1]{Vasiu:reconstructing}, we introduce the following definitions. For $q \in \mathbb{Z}_{>0}$, let $\alpha_M(q) \in \mathbb{Z}$ be the largest number such that $\varphi^q(M) \subset p^{\alpha_M(q)}M$ and let $\beta_M(q) \in \mathbb{Z}$ be the smallest number such that $p^{\beta_M(q)}M \subset \varphi^q(M)$; set $\delta_M(q) := \beta_M(q) - \alpha_M(q) \geq 0$. We note that if $(M, \varphi)$ is an $F$-crystal, then $\alpha_M(q), \beta_M(q) \geq 0$. It is not hard to prove that if $(M, \varphi)$ is isoclinic with Newton slope $\lambda$, then 
\begin{equation} \label{equation:alphabetafact1}
\alpha_M(q) \leq q\lambda \leq \beta_M(q),\; \forall \; q = 1, 2, \dots.
\end{equation}
Moreover we have 
\begin{equation} \label{equation:alphabetafact2}
\alpha_M(q) = q\lambda \; \;  \textrm{if and only if} \; \; \beta_M(q) = q\lambda.
\end{equation}
See \cite[Lemma 4.2.3]{Vasiu:reconstructing} for a proof of \eqref{equation:alphabetafact1} and \eqref{equation:alphabetafact2} in the Dieudonn\'e module case.

\begin{proposition}\label{proposition:computelii}
Let $(M,\varphi)$ be an isoclinic latticed $F$-isocrystal. Then $\ell_M = \max\{\delta_M(q) \; | \; q \in \mathbb{Z}_{> 0} \}$.
\end{proposition}
\begin{proof}
This proposition is a generalization of \cite[Proposition 4.3(a)]{Vasiu:reconstructing} and is proved in a similar way. For details, see \cite[Propositon 3.18]{Xiao:thesis}.
\end{proof}

\subsubsection{Computing the level torsion of a direct sum of isoclinic $F$-crystals} \label{subsection:computeldirectsum}
In this subsection, the latticed $F$-isocrystal $(M, \varphi) \cong \bigoplus_{i \in I} (M_i, \varphi_i)$ will always be a finite direct sum of isoclinic latticed $F$-isocrystals $(M_i, \varphi_i)$ with Newton slopes $\lambda_i$. For $i \in I$, let $B_i$ be a $W$-basis of $M_i$ and $B_i^*$ be the corresponding dual basis of $M_i^*$. Then 
\[B_i \otimes B_j^* := \{x \otimes y^* \; | \; x \in B_i, y^* \in B^*_j\}\] 
is a $W$-basis of $M_i \otimes M_j^* \cong \Hom(M_j, M_i)$. For each pair $(j,i) \in I \times I$ with $\lambda_j \leq \lambda_i$, define $\ell(j,i) \in \mathbb{Z}_{\geq 0}$ to be the smallest integer such that for all $q \in \mathbb{Z}_{>0}$ and all $x \otimes y^* \in B_i \otimes B_j^*$, we have $p^{\ell(j,i)}\varphi^q(x \otimes y^*) \in \Hom(M_j, M_i)$. We observe that $\ell(i,i) = \ell_{M_i}$. It is not hard to see that 
\[\ell_0 := \max\{\ell(j,i) \; | \; (j, i) \in I \times I, \lambda_j \leq \lambda_i \}\]
is the smallest non-negative integer $\ell$ such that $p^{\ell} \End(M) \subset O$. In most cases, we have $\ell_0 = \ell_M$ except when $O = \End(M)$ and $O_+\oplus O_-$ is not topologically nilpotent,  we have $\ell_M = 1$ and $\ell_0 = 0$. Therefore, we define an integer $\epsilon_M \in \{0,1\}$ by the following two rules to fix this problem:
\begin{enumerate}[(i)]
\item if $O = \End(M)$ and $O_+ \oplus O_-$ is not topologically nilpotent, then $\epsilon_M := 1$;
\item in all other cases, define $\epsilon_M:=0$.
\end{enumerate}

\begin{proposition} \label{proposition:computelisoclinic}
Let $(M, \varphi)$ and $\epsilon_M$ be as above, we have the following formula
\[ \ell_M = \max\{\epsilon_M, \ell(j,i) \; | \; (j, i) \in I \times I, \lambda_j \leq \lambda_i \}.\]
\end{proposition}
\begin{proof}
See {\cite[Scholium 3.5.1]{Vasiu:reconstructing}}.
\end{proof}
In general, it is easy to compute $\epsilon_M$. If $j=i$, then $\ell(j,i) = \ell_{M_i}$ can be computed by Proposition \ref{proposition:computelii}. If $j \neq i$, then we use the next proposition to compute $\ell(j,i)$.
\begin{proposition} \label{proposition:computeldirectsum}
Let $(M,\varphi) = \bigoplus_{i \in I}(M_i, \varphi_i)$ be a direct sum of two or more isoclinic latticed $F$-isocrystals. For $(j, i) \in I \times I$, $j \neq i$, and $\lambda_j \leq \lambda_i$, we have 
\[\ell(j,i) = \max\{0, \beta_{M_j}(q) - \alpha_{M_i}(q) \; | \; q \in \mathbb{Z}_{> 0} \}.\]
\end{proposition}
\begin{proof}
This proposition is a generalization of \cite[Proposition 4.4]{Vasiu:reconstructing}, and can be proved in a similar way. For details, see \cite[Proposition 3.21]{Xiao:thesis}.
\end{proof}
The next proposition uses the previous two propositions to estimate the isomorphism number of a direct sum of isoclinic latticed $F$-isocrystals.

\begin{proposition}\label{proposition:estimateldirectsum}
Let $(M,\varphi) = \bigoplus_{i \in I}(M_i, \varphi_i)$ be a direct sum of two or more isoclinic latticed $F$-isocrystals. Then we have the following inequality:
\[n_M \leq \max\{1, n_{M_i}, n_{M_i}+n_{M_j}-1 \; | \; i,j \in I, i \neq j \}. \]
\end{proposition}
\begin{proof}
This proposition is a generalization of \cite[Proposition 1.4.3]{Vasiu:reconstructing}.

As $(M,\varphi)$ is a direct sum of isoclinic latticed $F$-isocrystals, we have $n_M = \ell_M$ and $n_{M_i} = \ell_{M_i}$ for all $i \in I$ by Theorem \ref{theorem:vasiu1}. Hence it suffices to prove the proposition with all $n$ replaced by $\ell$. As $\epsilon_M \leq 1$, it suffices to show that $\ell(j,i) \leq \max\{0,\, \ell_{M_i} + \ell_{M_j} -1 \}$ if $j \neq i$ and $\lambda_j \leq \lambda_i$ by Proposition \ref{proposition:computelisoclinic}.  By Proposition \ref{proposition:computeldirectsum}, it suffices to prove that for all $q > 0$ we have
\begin{equation} \label{equation:estimateldirectsum1}
\beta_{M_j}(q) - \alpha_{M_i}(q) \leq \max\{0, \ell_{M_i}+\ell_{M_j}-1\}.
\end{equation}
As 
\begin{equation} \label{equation:estimateldirectsum2}
\alpha_{M_j}(q) \leq q \lambda_j  \leq q \lambda_i \leq \beta_{M_i}(q)
\end{equation}
 and $\delta_{M_i}(q) \leq \ell_{M_i}$, we have
\begin{align*}
\beta_{M_j}(q) - \alpha_{M_i}(q) &= \delta_{M_j}(q)+\alpha_{M_j}(q) + \delta_{M_i}(q) - \beta_{M_i}(q) \\ 
                               &\leq \ell_{M_j} + \ell_{M_i} + (\alpha_{M_j}(q) - \beta_{M_i}(q)) \\
                               &\leq \ell_{M_j} + \ell_{M_i}.
\end{align*}
In the case that the equality holds, necessarily $\alpha_{M_j}(q) - \beta_{M_i}(q) = 0$, whence $\alpha_{M_j}(q) = q\lambda_j = q\lambda_i = \beta_{M_i}(q)$ by \eqref{equation:estimateldirectsum2}. In particular, we have $\alpha_{M_j}(q) = \beta_{M_j}(q) = q\lambda_j$ as well as  $\alpha_{M_i}(q) = \beta_{M_i}(q) = q\lambda_i$ by \eqref{equation:alphabetafact2}. Therefore $\ell_{M_i} = 0 = \ell_{M_j}$ and $\beta_{M_j} - \alpha_{M_i} = \max\{0, \ell_{M_i}+\ell_{M_j}-1\}$, which proves \eqref{equation:estimateldirectsum1}.
\end{proof}

\subsection{Isoclinic ordinary $F$-crystals}

\begin{definition}
An $F$-crystal is called \emph{isoclinic ordinary} if its Hodge polygon is a straight line.  
\end{definition}

By Mazur's theorem \cite[Page 662, Lemma]{Mazur:FrobeniusHodge}, if the Hodge polygon is a straight line, then the Newton polygon, lying on or above the Hodge polygon with the same endpoints, is also a straight line. Thus isoclinic ordinary $F$-crystals are indeed isoclinic.

\begin{proposition} \label{proposition:n=0}
An $F$-crystal $(M, \varphi)$ is isoclinic ordinary if and only if $n_M = 0$.
\end{proposition}

\begin{lemma} \label{lemma:n=0}
Let $(M, \varphi)$ be an $F$-crystal of rank $r$ such that $\varphi(M) = M$. Then there is a $W$-basis $\{v_1, v_2, \dots, v_r\}$ of $M$ such that $\varphi(v_i) = v_i$ for $i = 1,2,\dots,r$. 
\end{lemma}
\begin{proof}
The lemma is an easy consequence of \cite[A.1.2.6]{Fontaine:padicrep1}. Using the same notation as \cite{Fontaine:padicrep1}, as $E = k$ is algebraically closed, we know that $\widehat{\mathcal{E}}_{\textrm{nr}} = \mathcal{E}_{\textrm{nr}} = \mathcal{E} = K_0$, whence the ring of integers $\mathcal{O}_{\widehat{\mathcal{E}}_{\textrm{nr}}} = \mathcal{O}_{\mathcal{E}} = W$. The Galois group $G_E = G_k$ is trivial as $k$ is algebraically closed. Let $M_0 := \{x \in M \; | \; \varphi(x) = x\}$ be the $\mathbb{Z}_p$-submodule of $M$ that contains all the elements fixed by $\varphi$. Applying the compositon of functors ${\bf D}_{\mathcal{E}} {\bf V}_{\mathcal{E}}$ to $(M, \varphi)$, we get that 
\[{\bf D}_{\mathcal{E}} {\bf V}_{\mathcal{E}} (M) = {\bf D}_{\mathcal{E}}((W \otimes_W M)_{\varphi = 1}) = {\bf D}_{\mathcal{E}}(M_0) = W \otimes_{\mathbb{Z}_p} M_0.\]
We know that $M = {\bf D}_{\mathcal{E}} {\bf V}_{\mathcal{E}} (M)$ by \cite[A.1.2.6]{Fontaine:padicrep1}, and thus $M = W \otimes_{\mathbb{Z}_p} M_0$. So choosing a $\mathbb{Z}_p$-basis of $M_0$ gives the desired result.
\end{proof}

\begin{proof}[Proof of Proposition \ref{proposition:n=0}]
Suppose $(M, \varphi)$ is isoclinic ordinary, we can assume that the Hodge polygon has slope $0$, namely $\varphi(M) = M$, by Remark \ref{remark:rescale}. For each $g \in \GL(M)$, we have $g\varphi(M) = M$, hence the Hodge polygon of the $F$-crystal $(M, g\varphi)$ is also a straight line of slope $0$. By Lemma \ref{lemma:n=0}, we get that $(M, g\varphi) \cong (M, \varphi)$ and thus $n_M=0$.

The converse is \cite[Lemma 2.3]{Vasiu:reconstructing}.
\end{proof}

\begin{corollary} \label{corollary:directsumordinary}
If $(M, \varphi)$ is a direct sum of two or more isoclinic ordinary $F$-crystals of distinct Hodge slopes, then $n_M = 1$.
\end{corollary}
\begin{proof}
By Proposition \ref{proposition:estimateldirectsum}, we have $n_M \leq 1$. If $n_M = 0$, then $(M, \varphi)$ is isoclinic ordinary by Proposition \ref{proposition:n=0}, which is a contradiction. Therefore $n_M = 1$.
\end{proof}

\section{Proof of Theorem \ref{theorem:estimate_n_elementary}} \label{section:proof}

Before we prove Theorem \ref{theorem:estimate_n_elementary}, we recall a lemma about the interrelation between the smallest Newton slope of an $F$-crystal and the smallest Hodge slope of the iterates of the $F$-crystal.

\begin{lemma}\label{lemma:sharpestimate}
Let $(M,\varphi)$ be an $F$-crystal, and let $\lambda \geq 0$ be a rational number. Let $h_0, h_1, \dots$ be the Hodge numbers of $(M,\varphi)$. Then all Newton slopes of $(M,\varphi)$ are greater than or equal to $\lambda$ if and only if for all integers $n> 0$, we have $\alpha_M(n+\sum_{i<\lambda}h_i) \geq \lceil n\lambda \rceil$.
\end{lemma}
\begin{proof}
See \cite[Section 1.5]{Katz:slopefiltration}.
\end{proof}

We set some notations that will be useful in the proof of Theorem \ref{theorem:estimate_n_elementary}. Let $(M, \varphi)$ be an $F$-crystal. By definition, we have $p^e M \subset \varphi(M)$ where $e:=\beta_M(1)$. Thus $\varphi(M^*) \subset p^{-e}M^*$, i.e. $p^{e} \varphi(M^*) \subset M^*$ (Recall that $(M^*, \varphi)$ is the dual of $(M, \varphi)$ and is not an $F$-crystal if $e > 0$). Therefore $(M',\varphi'):= (M^*, p^{e}\varphi)$ is an $F$-crystal.

\begin{remark} \label{remark:doubleduality}
As the isomorphism number of $(M, \varphi)$ is equal to the isomorphism number of $(M^*, \varphi)$, and the isomorphism number of $(M^*, \varphi)$ is equal to the isomorphism number of $(M^*, p^e\varphi)$ by Remark \ref{remark:rescale}, the isomorphism number of $(M, \varphi)$ is equal to the isomorphism number of $(M', \varphi')$.
\end{remark}

\begin{lemma} \label{lemma:mainlemma1}
Let $(M, \varphi)$ be an $F$-crystal with $\alpha_M(1) = 0$ and $e:=\beta_M(1) >0$. Let $(M',\varphi')$ be as above. If $h_i$ and $h_i'$ are the Hodge numbers of $(M, \varphi)$ and $(M', \varphi')$ respectively, then for any $\lambda \in (0, e)$, we have 
\[\sum_{i<e - \lambda} h_i' = \sum_{i > \lambda} h_i.\]
\end{lemma}
\begin{proof}
 Let $0 = e_1 \leq e_2 \leq \cdots \leq e_r = e$ be the Hodge slopes of $(M, \varphi)$. Then there are $W$-bases $\{v_1, \dots, v_r\}$ and $\{w_1, \dots, w_r\}$ of $M$ such that $\varphi(v_i) = p^{e_i}w_i$ for all $i =1,2,\dots,r$. Let $\{v_1^*, \dots, v_r^*\}$ and $\{w_1^*, \dots, w_r^*\}$ be the corresponding dual $W$-basis of $M^* = M'$, so $\varphi(v_i^*) = p^{-e_i}w_i^*$. Multiplying by $p^e$, we have $\varphi'(v_i^*) = p^e\varphi(v_i^*) = p^{e-e_i}w_i^*$. This means that 
\[0=e-e_r \leq e-e_{r-1} \leq \cdots \leq e-e_1=e\] 
are the Hodge slopes of $(M', \varphi')$. For $j =1,2,\dots,r$, set $e_{r-j+1}' := e_r-e_j$, i.e. $e_j + e'_{r-j+1} = e_r$. If $e_j > \lambda$, then $e'_{r-j+1} = e_r - e_j < e - \lambda$; if $e'_i < e_r - \lambda$, then $e_{r-i+1} = e_r - e'_i > e_r - (e_r - \lambda) = \lambda$. This describes a bijection between the sets $\{e_j \; | \; e_j > \lambda \}$ and $\{e'_i \; | \; e'_i < e - \lambda \}$, whence the lemma.
\end{proof}

\begin{lemma} \label{lemma:mainlemma2}
Let $(M, \varphi)$ be an $F$-crystal and let $(M', \varphi')$ be as above. Suppose $\alpha_M(1)=0$ and set $e:=\beta_M(1)$. Then for $q =1,2,\dots$, we have 
\[\alpha_M(q) + \beta_{M'}(q) = qe = \alpha_{M'}(q) + \beta_M(q).\]
\end{lemma}
\begin{proof}
By the definitions of $\alpha_M(q)$ and $\beta_M(q)$, we have
\[p^{\beta_M(q)} M \subset \varphi^{q}(M) \subset p^{\alpha_M(q)} M,\] 
and thus
\[p^{-\alpha_M(q)}M' \subset \varphi^q(M') \subset p^{-\beta_M(q)}M'.\] 
Multiplying by $p^{qe}$, we have 
\[p^{qe - \alpha_M(q)}M' \subset (p^{e}\varphi)^q(M') \subset p^{qe - \beta_M(q)}M',\] 
hence
\begin{equation} \label{equation:alphabeta1}
\beta_{M'}(q) \leq qe-\alpha_M(q) \quad \textrm{and} \quad \alpha_{M'}(q) \geq qe-\beta_M(q).
\end{equation}
Again by the definitions of $\alpha_{M'}(q)$ and $\beta_{M'}(q)$, we have
\[p^{\beta_{M'}(q)}M' \subset (\varphi')^q(M') \subset p^{\alpha_{M'}(q)}M',\]
and thus
\[p^{-\alpha_{M'}(q)}M \subset (\varphi')^q(M) \subset p^{-\beta_{M'}(q)}M,\]
that is
 \[p^{-\alpha_{M'}(q)}M \subset p^{-qe}\varphi^q(M) \subset p^{-\beta_{M'}(q)}M.\]
Multiplying by $p^{qe}$, we obtain 
\[p^{qe - \alpha_{M'}(q)}M \subset \varphi^q(M') \subset p^{qe - \beta_{M'}(q)}M,\]
and hence
\begin{equation} \label{equation:alphabeta2}
\alpha_M(q) \geq qe - \beta_{M'}(q) \quad \textrm{and} \quad \beta_M(q) \leq qe-\alpha_{M'}(q).
\end{equation} 
Lemma \ref{lemma:mainlemma2} is now clear by inequalities \eqref{equation:alphabeta1}, \eqref{equation:alphabeta2}.
\end{proof}

\begin{proof}[\textbf{Proof of Theorem \ref{theorem:estimate_n_elementary}}]
If $e=0$, then $\lambda=0$ and the $F$-crystal $(M, \varphi)$ is isoclinic ordinary. By Proposition \ref{proposition:n=0}, we get $n_M=0$. In this case, inequality $\eqref{equation:main}$ is in fact an equality as both sides are equal to $0$. Now we can assume that $e>0$ and thus $\lambda<e$.

To ease notation, let $l_1 = \sum_{i<\lambda}h_i$ and $l_2 = \sum_{i>\lambda}h_i$. To prove the inequality \eqref{equation:main}, it suffices to prove that 
\begin{equation} \label{equation:main2}
\ell_M \leq el_2+(l_1-l_2)\lambda
\end{equation}
by Theorem \ref{theorem:vasiu1}. By Proposition \ref{proposition:computelii}, it suffices to prove that for all $q \in \mathbb{Z}_{>0}$,
\begin{equation} \label{equation:main1}
\delta_M(q) \leq el_2+(l_1-l_2)\lambda.
\end{equation}
By Lemma \ref{lemma:sharpestimate}, we have $\alpha_M(q) \geq \lceil (q-l_1) \lambda \rceil$ for all $q > l_1$. If $q \leq l_1$, as $(M, \varphi)$ is an $F$-crystal, we have $\alpha_M(q) \geq 0 \geq \lceil (q-l_1)\lambda \rceil$. Thus for all $q > 0$, we have
\[\alpha_M(q) \geq \lceil (q-l_1) \lambda \rceil.\]
To find an upper bound for $\beta_M(q)$, let $(M', \varphi')$ be the $F$-crystal $(M^*, p^{e}\varphi)$. It is isoclinic with Newton slope equal to $e - \lambda > 0$. If $h'_i$ are the Hodge numbers of $(M', \varphi')$, then by Lemma \ref{lemma:sharpestimate}, for all $q > \sum_{i < e - \lambda} h'_i$, we have 
\begin{equation} \label{equation:alpha}
\alpha_{M'}(q) \geq \lceil (q - \sum_{i < e - \lambda} h'_i)(e - \lambda) \rceil.
\end{equation}
By Lemma \ref{lemma:mainlemma1}, we have $\alpha_{M'}(q) = qe - \beta_M(q)$. By Lemma \ref{lemma:mainlemma2}, we have $\sum_{i<e-\lambda} h_i' = l_2$. Therefore inequality \eqref{equation:alpha} becomes $qe - \beta_M(q) \geq \lceil (q - l_2)(e - \lambda) \rceil$
for all $q > l_2$. On the other hand, if $0 < q \leq l_2$, as $p^eM \subset \varphi(M)$, we have $p^{qe}M \subset \varphi^q(M)$, which means that   $qe - \beta_M(q) \geq 0 \geq \lceil (q - l_2)(e - \lambda) \rceil$. Thus for all $q > 0$, we have 
\[qe - \beta_M(q) \geq \lceil (q - l_2)(e - \lambda) \rceil.\]
We are now ready to find an upper bound for $\beta_M(q)$.
For all $q >0$, we have
\begin{gather} 
\beta_M(q) \leq qe - \lceil (q - l_2)(e - \lambda) \rceil =qe - \lceil (q - l_2)e - (q - l_2)\lambda \rceil \notag \\
           =qe - (q - l_2)e + \lfloor (q - l_2)\lambda \rfloor =    el_2 + \lfloor (q - l_2)\lambda \rfloor \label{equation:beta} .
\end{gather}
Combining the estimates \eqref{equation:alpha} and \eqref{equation:beta}, for all $q > 0$,  we have
\begin{gather}
\delta_M(q) = \beta_M(q) - \alpha_M(q) \leq  el_2 + \lfloor (q - l_2)\lambda \rfloor - \lceil (q - l_1) \lambda \rceil \notag\\
            \leq  el_2 + (q - l_2)\lambda - (q - l_1) \lambda = el_2 + (l_1 - l_2)\lambda. \notag
\end{gather}
Thus, inequality \eqref{equation:main1} holds and the proof of the theorem is complete.
\end{proof}

\begin{corollary} \label{corollary:recover}
Let $(M, \varphi)$ be the Dieudonn\'e module corresponding to an isoclinic $p$-divisible group $D$ with dimension $d$ and codimension $c$, then the following inequality holds $n_M \leq \lfloor 2cd/(c+d) \rfloor$.
\end{corollary}
\begin{proof}
Since $D$ is isoclinic, the Dieudonn\'e module $(M, \varphi)$ is also isoclinic. The Newton slope is $\lambda = d/(c+d)$ and the Hodge numbers are $h_0 = c$, $h_1 = d$ and $h_i = 0$ for all $i > 1$. The Hodge slopes are $e_1 = \cdots = e_c = 0$ and $e_{c+1} = \cdots = e_{c+d} = 1$. By Theorem \ref{theorem:estimate_n_elementary}, we have 
\[n_M \leq \lfloor d + d(c-d)/(c+d) \rfloor = \lfloor 2cd/(c+d) \rfloor. \qedhere\]
\end{proof}

\begin{example} \label{example:supersingular}
Consider an isoclinic $F$-crystal $(M, \varphi)$ of rank $r = 2d$, $d \in \mathbb{Z}_{>0}$ with Hodge slopes $e_i = 0$ if $1 \leq i \leq d$ and $e_i = e > 0$ if $d+1 \leq i \leq r$. The unique Newton slope is equal to $e/2$. By Theorem \ref{theorem:estimate_n_elementary}, the isomorphism number is $n_M \leq de$. In fact, this inequality is optimal in the sense that there exists an isoclinic $F$-crystal $(M, \varphi)$ with the above rank and Hodge slopes such that $n_M = de$; see Proposition \ref{proposition:permutational}. This type of $F$-crystal is a generalization of supersingular Dieudonn\'e modules, (cf. \cite{Vasiu:supersingular}) which correspond to the case $e=1$.
\end{example}

\begin{example} \label{example:est_n_ele_not_optimal}
Let $(M, \varphi)$ be an isoclinic $F$-crystal of rank $3$ with Hodge slopes $e_1=0, e_2=1, e_3=5$ and Newton slope $\lambda = 2$. By the analysis of the Hodge slopes of the iterates of $(M, \varphi)$ using elementary row and column operations, it can be shown that $n_M \leq 6$. The details are a bit messy and omitted here. On the other hand, by using Theorem \ref{theorem:estimate_n_elementary}, we get that $n_M \leq 7$. This implies that Theorem \ref{theorem:estimate_n_elementary} can be improved in some cases and is not optimal in general.
\end{example}

\section{Applications} \label{section:applications}

\subsection{$F$-crystals of K3 type} \label{section:K3}

We recall that an $F$-crystal $(M, \varphi)$ of rank $r \in \mathbb{Z}_{\geq 2}$ is of K3 type if its Hodge numbers are $h_0=1, h_1=r-2, h_2=1$ and $h_i = 0$ for all $i \geq 3$. By Mazur's theorem \cite[Page 662, Lemma]{Mazur:FrobeniusHodge}, it can be shown that there are $(r^2-r+2)/2$ possible Newton polygons for $F$-crystals of K3 type. In fact, each possible Newton polygon is indeed the Newton polygon of some $F$-crystal of K3 type by a theorem of Kottwitz and Rapoport \cite[Theorem A]{KottwitzRapoport:existence}. If an $F$-crystal of K3 type is isoclinic, then all of its Newton slopes are equal to $1$. If it is non-isoclinic, then the Newton slopes could be in one of the following two disjoint cases:
\begin{enumerate}[(a)]
\item $r_1/(r_1+1)$, $1$, and $(r_2+2)/(r_2+1)$ if $r_1$ and $r_2$ satisfy $r_1, r_2 > 0$ and $0 < r_1+r_2 < r-2$, or
\item $r_1/(r_1+1)$ and $(r_2+2)/(r_2+1)$ if $r_1, r_2 >0$ and $r_1+r_2 = r-2$.
\end{enumerate}

\begin{proposition} \label{proposition:K3nonisoclinic}
If $(M, \varphi)$ is a non-isoclinic $F$-crystal of K3 type, then $n_M = 1$.
\end{proposition}
\begin{proof}
By \cite[Section 1.6]{Katz:slopefiltration}, we have a direct sum decomposition
\[(M,\varphi) \cong (M_1, \varphi_1) \oplus (M_2, \varphi_2) \oplus (M_3, \varphi_3),\] 
where
\begin{itemize}
\item $(M_1, \varphi_1)$ has Hodge numbers $h_0=1, h_1=r_1$, $h_i = 0$ for all $i \in \mathbb{Z}_{\geq 2}$ and Newton slope $r_1/(r_1+1)$. By \cite[Page 92, Proposition]{Demazure1}, the $W$-module $M_1$ has a $W$-basis $B_1 = \{x_1, \dots, x_{r_1+1} \}$ such that 
\[\varphi_1(x_i)=px_{i+1}, \; \forall \; i=1,2,\dots,r_1; \qquad \varphi_1(x_{r_1+1})=x_1.\]
\item $(M_2, \varphi_2)$ has Hodge numbers $h_1=r-r_1-r_2-2$, $h_i = 0$ for $i=0, 2, 3, \dots$. and Newton slope $1$. Hence $\varphi_2(M_2) = pM_2$. Applying  Lemma \ref{lemma:n=0} to $(M_2, p^{-1}\varphi_2)$, we get a $W$-basis $B_2 = \{y_1, \dots, y_{r-r_1-r_2-2} \}$ of $M_2$ such that $p^{-1}\varphi_2(y_i)=y_i$, and thus 
\[\varphi_2(y_i)=py_i, \; \forall \; i=1,2,\dots,r-r_1-r_2.\] 
\item $(M_3, \varphi_3)$ has Hodge numbers $h_1=r_2, h_2=1$, $h_i = 0$ for $i = 0, 3,4\dots$ and Newton slope $(r_2+2)/(r_2+1)$. Applying \cite[Page 92, Proposition]{Demazure1} to $(M_3, p^{-1}\varphi_3)$ whose Newton slope is $1/(r_2+1)$, we get a $W$-basis $B_3 = \{z_1, \dots, z_{r_2+1}\}$ of $M_3$ such that 
\[\varphi_3(z_i)=pz_{i+1}, \; \forall \; i=1,2,\dots,r_2; \qquad \varphi_3(z_{r_2+1})=p^2z_1.\]
\end{itemize}
We first calculate $\ell_M$ for $(M, \varphi)$ in Case (a) where $M_2 \neq 0$. The Case (b) where $M_2 = 0$ will be handled later.

We use Proposition \ref{proposition:computelisoclinic} to compute $\ell_M$. First we compute $\ell_{M_1}$, $\ell_{M_2}$, and $\ell_{M_3}$. Since $\delta_{M_1}(q)=1$ for all $q \in \mathbb{Z}_{\geq 1} \backslash \{n(r_1+1) \; | \; n \in \mathbb{Z}_{>0}\}$ and $\delta_{M_1}(n(r_1+1)) = 0$ for all $n \in \mathbb{Z}_{> 0}$, we know that $\ell_{M_1} = 1$ by Proposition \ref{proposition:computelii}. By the same token, we have $\ell_{M_3} = 1$. Since $\varphi_2^q(M) = p^qM$ for all $q \in \mathbb{Z}_{> 0}$, we know that $\ell_{M_2} = 0$ by Proposition \ref{proposition:computelii}. 

Next, we compute $\ell(1,2)$, $\ell(2,3)$, and $\ell(1,3)$. For $x_i \in B_1$ and $y_j \in B_2$, we have
\begin{equation*}
\varphi(y_j \otimes x_i^*) = \left\{
\begin{array}{ll}
y_j\otimes x_{i+1}^* & \textrm{if}\; 1 \leq i \leq r_1\\
py_j\otimes x_{1}^*  & \textrm{if}\; i=r_1+1.
\end{array} \right.
\end{equation*}
Hence $\ell(1,2) = 0$. For $y_j \in B_2$ and $z_l \in B_3$, we have
\begin{equation*}
\varphi(z_l \otimes y_j^*) = \left\{
\begin{array}{ll}
z_{l+1} \otimes y_j^* & \textrm{if}\;  1 \leq l \leq r_2\\
pz_1 \otimes y_j^*  &  \textrm{if}\;  l=r_2+1.
\end{array} \right.
\end{equation*}
Hence $\ell(2,3) = 0$. For $x_i \in B_1$ and $z_l \in B_3$, we have
\begin{equation*}
\varphi(z_l \otimes x_i^*) = \left\{
\begin{array}{ll}
z_{l+1} \otimes x_{i+1}^* & \textrm{if}\;  1 \leq i \leq r_1, 1\leq l \leq r_2\\
pz_1 \otimes x_{i+1}^* & \textrm{if} \; 1 \leq i \leq r_1, l=r_2+1 \\
pz_{l+1} \otimes x_{1}^*  &  \textrm{if}\; i=r_1+1, 1\leq l \leq r_2\\
p^2z_1 \otimes x_{1}^* & \textrm{if} \; i=r_1+1, l=r_2+1.
\end{array} \right.
\end{equation*}
Hence $\ell(1,3) = 0$. By Proposition \ref{proposition:computelisoclinic}, we have
\[\ell_M = \max\{\epsilon_M,\ell_{M_1}, \ell_{M_2}, \ell_{M_3}, \ell(1,2), \ell(2,3), \ell(1,3) \} = 1.\]

In Case (b) where $M_2=0$, we have $\ell_M = \max\{\epsilon_M,\ell_{M_1}, \ell_{M_3}, \ell(1,3) \} = 1.$ Thus in both Case (a) and Case (b), we have $\ell_M = 1$.

By Theorem \ref{theorem:vasiu1}, we have $n_M \leq \ell_M = 1$. On the other hand, the $F$-crystal $(M, \varphi)$ is not an ordinary $F$-crystal and thus $n_M \neq 0$ by Proposition \ref{proposition:n=0}. Hence $n_M = 1$.
\end{proof}

\begin{corollary} \label{corollary:K3directsumnoniso}
Let $(M, \varphi)$ be a direct sum of two or more non-isoclinic $F$-crystals of K3 type, then $n_M = 1$.
\end{corollary}
\begin{proof}
By Propositions \ref{proposition:estimateldirectsum} and \ref{proposition:K3nonisoclinic}, we have $n_M \leq 1$. As $(M, \varphi)$ is not isoclinic ordinary, we have $n_M \neq 0$. Hence $n_M = 1$.
\end{proof}

\begin{proposition} \label{proposition:K3isoclinic}
Let $(M, \varphi)$ be an isoclinic $F$-crystal of K3 type. Then $n_M = 2$.
\end{proposition}
\begin{proof}
The unique Newton slope of $(M, \varphi)$ is $1$. The largest Hodge slope is $2$ and $\sum_{i < 1} h_i = \sum_{i > 1} h_i = 1$. By Theorem \ref{theorem:estimate_n_elementary}, we have $n_M \leq 2$. As $\ell_M \geq \delta_M(1) = e_r - e_1 = 2$ by Proposition \ref{proposition:computelii}, we conclude that $n_M = \ell_M = 2$.
\end{proof}

\begin{corollary} \label{corollary:K3directsumiso}
Let $(M, \varphi)$ be a direct sum of two or more isoclinic $F$-crystals of K3 type, then $n_M = 2$.
\end{corollary}
\begin{proof}
Let $(M,\varphi) \cong \bigoplus_{i \in I} (M_i, \varphi_i)$ be a finite direct sum of two or more isoclinic $F$-crystals of K3 type. The Newton slopes of $(M_i, \varphi_i)$ are $1$ for all $i \in I$. Hence $(M,\varphi)$ is again isoclinic (but not of K3 type). Thus we can use Proposition \ref{proposition:computelisoclinic} to compute $\ell_M$. For each $i \in I$, we know that $n_{M_i} = \ell_{M_i} = 2$ by Proposition \ref{proposition:K3isoclinic}. To calculate $\ell(j,i)$, we use Proposition \ref{proposition:computeldirectsum}. By Lemma \ref{lemma:sharpestimate}, we have $\alpha_{M_i}(q) \geq q-1$ for $q = 1, 2, \dots$ and for all $i \in I$. Let $(M_j', \varphi_j')$ be the $F$-crystal $(M_j^*, p^2\varphi_j)$ for all $j \in I$. By Lemma \ref{lemma:mainlemma2}, we have $\beta_{M_j}(q) = 2q-\alpha_{M'_j}(q)$. Applying Lemma \ref{lemma:sharpestimate} to $(M_j', \varphi_j')$, we have $\alpha_{M_j'}(q) \geq q-1$. We conclude that $\beta_{M_j}(q) \leq q+1$. Hence $\beta_{M_j}(q) - \alpha_{M_i}(q) \leq (q+1) - (q-1) \leq 2$ for all $q \in \mathbb{Z}_{>0}$ and $i,j\in I$. By Proposition \ref{proposition:computeldirectsum}, we have
\begin{equation} \label{equation:K3ldifferent}
\ell(j,i) = \max\{0, \beta_{M_j}(q) - \alpha_{M_i}(q) \; | \; q \in \mathbb{Z}_{>0}\} \leq 2.
\end{equation}
Since $\epsilon_M \leq 1$ and $\ell_{M_i} = 2$, inequality \eqref{equation:K3ldifferent} and Proposition \ref{proposition:computelisoclinic} imply that 
\[\ell_M = \max\{\epsilon_M, \ell_{M_i}, \ell(j,i) \; | \; i, j \in I \} = \max\{\ell_{M_i} \, | \; i \in I\} = 2.\] 
By Theorem \ref{theorem:vasiu1}, we have $n_M = \ell_M = 2$.
\end{proof}

\begin{proof}[\textbf{Proof of Theorem \ref{theorem:K3}}]
Suppose $(M,\varphi) = \bigoplus_{i \in I} (M_i, \varphi_i)$ is a mixed direct sum of isoclinic and non-isoclinic $F$-crystals of K3 type. Let $(M_{\textrm{iso}}, \varphi_{\textrm{iso}})$ be the direct sum of all isoclinic ones. By Corollary \ref{corollary:K3directsumiso}, we know that $n_{M_{\textrm{iso}}} = 2$. Every non-isoclinic $F$-crystal of K3 type can be decomposed into three isoclinic $F$-crystals (not of K3 type) whose isomorphism numbers are less than or equal to $1$; see proof of Proposition \ref{proposition:K3nonisoclinic}. By Proposition \ref{proposition:estimateldirectsum}, we have $n_M \leq 2$.

Parts (i) and (ii) have been proved by Corollaries \ref{corollary:K3directsumnoniso} and \ref{corollary:K3directsumiso} respectively. For Part (iii), if $(M_i, \varphi_i)$ is isoclinic and a direct summand of $(M, \varphi)$, then $n_M = \ell_M \geq \ell_{M_i} = n_{M_i} = 2$ by Theorem \ref{theorem:vasiu1} and Proposition \ref{proposition:computelisoclinic}. As $n_M \leq 2$ in general, we have $n_M = 2$ in this case.
\end{proof}

The \emph{isogeny cutoff} $b_M$ of an $F$-crystal $(M, \varphi)$ is the smallest non-negative integer such that for every $g \in \GL(M)$ with $g \equiv 1$ mod $p^{b_M}$, the $F$-crystal $(M, g\varphi)$ has the same Newton polygon as $(M, \varphi)$. As $b_M \leq n_M$, it is also finite.
\begin{proposition}
Let $(M, \varphi)$ be an $F$-crystal of K3 type. Then $b_M = 1$.
\end{proposition}
\begin{proof}
If $(M, \varphi)$ is a non-isoclinic $F$-crystal of K3 type, we have proved that $n_M \leq 1$ and hence $b_M \leq n_M \leq 1$. If $(M, \varphi)$ is an isoclinic $F$-crystal of K3 type, then for any $g \in \GL(M)$ with the property that $g \equiv 1$ mod $p$, we have $M/\varphi(M) \cong M/g\varphi(M)$ as $W$-modules and thus $(M, g\varphi)$ and $(M, \varphi)$ have the same Hodge slopes, whence $(M, g\varphi)$ is also an $F$-crystal of K3 type. If $(M, g\varphi)$ is not isoclinic, then it is one of those non-isoclinic $F$-crystals of K3 type with isogeny cutoff less than or equal to $1$. From this and the fact that $g^{-1} \equiv 1$ mod $p$, we know that $(M, \varphi)$ is non-isoclinic, which is a contradiction. Thus $(M, g\varphi)$ is isoclinic and necessarily has the same Newton polygon as $(M, \varphi)$. This implies that $b_M \leq 1$ when $(M, \varphi)$ is isoclinic.

Next we prove that $b_M > 0$. Let $(M, \varphi)$ be an isoclinic $F$-crystal of K3 type. By \cite[Theorem A]{KottwitzRapoport:existence}, we know that there exists $g \in \GL(M)$ such that $(M, g\varphi)$ is non-isoclinic. Therefore $(M, \varphi)$ and $(M, g\varphi)$ do not have the same Newton polygon, and this proves that $b_M > 0$ if $(M, \varphi)$ is isoclinic. By the same token, we can show that $b_M > 0$ if $(M, \varphi)$ is non-isoclinic of K3 type. Therefore $b_M > 0$. As a result, we have $b_M = 1$.
\end{proof}

\subsection{$F$-crystals of rank $2$} \label{section:rank2}

In this section, we compute the isomorphism number of $F$-crystals of rank $2$. Without loss of generality, we can assume that the smallest Hodge slope is $0$ by Remark \ref{remark:rescale}. Let $e \geq 0$ be the other Hodge slope. If $e=0$, then the isomorphism number is zero by Proposition \ref{proposition:n=0}. Thus we assume that $e>0$. 

\begin{proof}[\textbf{Proof of Theorem \ref{theorem:rank2}}]
Let $\lambda_2$ be the other Newton slope of $(M, \varphi)$. 

We prove (i). If $(M, \varphi)$ is a direct sum of two $F$-crystals of rank $1$, then each direct summand of $(M, \varphi)$ is an $F$-crystal whose Hodge polygon and Newton polygon coincide. Therefore, the Hodge and Newton slopes of each direct summand are equal. Hence $\lambda_1 = e_1 = 0$ and $\lambda_2 = e_2 = e$. By Corollary \ref{corollary:directsumordinary}, we have $n_M = 1$.

We prove (ii). If $(M, \varphi)$ is not a direct sum of two $F$-crystals of rank $1$ and is isoclinic, then $\lambda_1 = \lambda_2 = e/2$. By Theorem \ref{theorem:estimate_n_elementary}, we have $n_M \leq e$. As $\ell_M \geq e$ by Proposition \ref{proposition:computelii}, we have $n_M = \ell_M = e$ by Theorem \ref{theorem:vasiu1}, as desired.

We prove (iii). If $(M,\varphi)$ is not a direct sum of two $F$-crystals of rank $1$ and is non-isoclinic, then the Newton slopes $\lambda_1 < \lambda_2$ are both positive integers. Indeed, if either $\lambda_1$ or $\lambda_2$ is not an integer, say $\lambda_1 = c/d \notin \mathbb{Z}$ (in reduced form), then $d$ must be $2$  as the number of times that $\lambda_1 = c/d$ appears as a Newton slope is a multiple of $d$. As there are only two Newton slopes, we know that $\lambda_1 = \lambda_2 \in \mathbb{Z}+1/2$. This contradicts to the fact that $(M, \varphi)$ is non-isoclinic. If $\lambda_1 = 0$, then $\lambda_2 = e$ which implies that $(M, \varphi)$ is a direct sum of two $F$-crystals of rank $1$. This is a contradiction again!

Now we assume that $0 < \lambda_1 < \lambda_2$ are both integers. There exists a $W$-basis $B_1 = \{x_1, x_2\}$ of $M$ such that $\varphi$ is of the form $\begin{pmatrix} p^{\lambda_1} & u \\ 0 & p^{\lambda_2} \end{pmatrix}$ where $u$ is a unit in $W$. By solving equations of the form $\varphi(z) = p^{\lambda_1}z$ and $\varphi(z) = p^{\lambda_2}z$, we find a $K_0$-basis $B_2 = \{y_1 = x_1, y_2 = vx_1+p^{\lambda_1}x_2\}$ of $M[1/p]$ with $v$ a unit in $W$ such that $\sigma(v)+u=p^{\lambda_2-\lambda_1}v$. As in Section \ref{subsection:computeldirectsum}, the set $B_1 \otimes B_1^*$ is a $W$-basis of $\End(M)$ and hence a $K_0$-basis of $\End(M[1/p])$; the set $B_2 \otimes B_2^*$ is another $K_0$-basis of $\End(M[1/p])$. As $\varphi(y_1)=p^{\lambda_1}y_1$ and $\varphi(y_2) = p^{\lambda_2}y_2$, we have 
\begin{align*}
\varphi(y_2 \otimes y^*_1) &= p^{\lambda_2-\lambda_1} y_2 \otimes y^*_1, \quad & \varphi(y_1 \otimes y^*_1) &= y_1 \otimes y^*_1, \\
\varphi(y_2 \otimes y^*_2) &= y_2 \otimes y^*_2, \quad &\varphi(y_1 \otimes y^*_2) &= p^{\lambda_1-\lambda_2} y_1 \otimes y_2^*.
\end{align*}
Therefore, we have found $K_0$-bases for
\[L_+ = \langle y_2 \otimes y^*_1 \rangle_{K_0}, \quad L_0 = \langle y_1 \otimes y^*_1, y_2 \otimes y^*_2 \rangle_{K_0}, \quad L_- = \langle y_1 \otimes y^*_2 \rangle_{K_0}.\]
We compute the change of basis matrix from $B_1 \otimes B_1^*$ to $B_2 \otimes B_2^*$ as follows:
\begin{align*}
y_1 \otimes y_1^* &= x_1 \otimes x^*_1-\frac{\sigma(v)}{p^{\lambda_1}}x_1 \otimes x_2^*,\\
y_2 \otimes y_1^* &= vx_1 \otimes x_1^*+p^{\lambda_1}x_2 \otimes x_1^*-\frac{\sigma(v)v}{p^{\lambda_1}}x_1 \otimes x_2^*-\sigma(v)x_2 \otimes x_2^*, \\
y_1 \otimes y^*_2 &= \frac{1}{p^{\lambda_1}}x_1 \otimes x^*_2,\\
y_2 \otimes y^*_2 &= \frac{v}{p^{\lambda_1}}x_1 \otimes x^*_2 + x_2 \otimes x^*_2.
\end{align*}
It is easy to see that $p^{\lambda_1} y_i \otimes y^*_j \in \End(M) \; \backslash \; p\End(M)$ for $i, j \in \{1,2\}$. We get that 
\begin{enumerate}[(a)]
\item $O_+ = \langle p^{\lambda_1}y_2 \otimes y_1 \rangle_{W}$; 
\item $N := \langle p^{\lambda_1}y_1 \otimes y^*_1, p^{\lambda_1}y_2\otimes y^*_2\rangle_W \subset O_0$ is a lattice; 
\item $O_- = \langle p^{\lambda_1}y_1 \otimes y^*_2 \rangle_{W}$. 
\end{enumerate}
Therefore, $O_+ \oplus N \oplus O_-$ is a sublattice of $O$. The change of basis matrix from $\{p^{\lambda_1}y_1 \otimes y_1^*, p^{\lambda_1}y_2 \otimes y_1^*, p^{\lambda_1}y_1 \otimes y_2^*, p^{\lambda_1}y_2 \otimes y_2^*\}$ to $B_1 \otimes B_1^*$ is
\[A = \begin{pmatrix} p^{\lambda_1} & p^{\lambda_1}v & 0 & 0 \\ 0 & p^{2\lambda_1} & 0 & 0 \\ -\sigma(v) & -\sigma(v)v & 1 & v \\ 0 & -p^{\lambda_1}\sigma(v) & 0 & p^{\lambda_1} \end{pmatrix} \]
To find a upper bound for $\ell_M$, we compute the inverse of $A$:
\[A^{-1} = \frac{1}{p^{2\lambda_1}}\begin{pmatrix} p^{\lambda_1} & -v & 0 & 0 \\ 0 & 1 & 0 & 0 \\ \sigma(v)p^{\lambda_1} & \sigma(v)v & p^{2\lambda_1} & -p^{\lambda_1}v \\ 0 & \sigma(v) & 0 & p^{\lambda_1} \end{pmatrix} \]
Thus the smallest number $\ell$ such that all entries of $p^{\ell}A^{-1} \in W$ is $2\lambda_1$. Hence $\ell_M \leq 2\lambda_1$. By Theorem \ref{theorem:vasiu1}, we have $n_M \leq 2\lambda_1$.
\end{proof}

\section{Quasi-special $F$-crystals}
\begin{lemma} \label{lemma:smust}
In Definition \ref{definition:quasi-special} of isoclinic quasi-special $F$-crystals, the non-negative number $s$ must equal to the sum of all Hodge slopes.
\end{lemma}
\begin{proof}
Consider the iterate $(M, \varphi^r/p^s)$; its Hodge polygon is a straight line of slope $0$. By Lemma \ref{lemma:n=0}, we know that there is a $W$-basis $\{v_1, v_2, \dots, v_r\}$ of $M$ such that $(\varphi^r/p^s)(v_i) = v_i$ and thus $\varphi^r(v_i) = p^sv_i$ for $i=1,2,\dots$. By the Dieudonn\'e-Manin classification of $F$-crystals up to isogeny, we know that every Newton slope must be equal to $s/r$. The sum of all Hodge slopes, which is equal to the sum of all Newton slopes, is equal to $\sum_{r} s/r = s$.
\end{proof}

\begin{lemma} \label{lemma:computelquasispecial}
If $(M, \varphi)$ is an isoclinic quasi-special $F$-crystal, then \[n_M = \max\{\delta_{M}(j) \; | \; j = 1, 2, \dots, r\}.\] 
\end{lemma}
\begin{proof}
As $(M, \varphi)$ is isoclinic, we have $n_{M} = \ell_{M}$ by Theorem \ref{theorem:vasiu1}. By definition, we have $\varphi^{r}(M) = p^{s}M$. This means that $\alpha_{M}(r) = \beta_{M}(r) = s$. In addition, for all $j \in \mathbb{Z}_{>0}$, we have $\alpha_{M}(r+j) = \alpha_{M}(j)+s$ and $\beta_{M}(r+j) = \beta_{M}(j)+s$, thus $\delta_M(r+j) = \delta(j)$. Hence by Proposition \ref{proposition:computelii}, the Lemma follows from:
\[n_M = \ell_{M} = \max\{\delta_{M}(j) \; | \; j \in \mathbb{Z}_{> 0}\} = \max\{\delta_{M}(j) \; | \; j = 1, 2, \dots, r\}. \qedhere\]
\end{proof}

\begin{proof}[\textbf{Proof of Theorem \ref{theorem:quasi-special}}]
Let $(M, \varphi) = \bigoplus_{i = 1}^t (M_i, \varphi_i)$ be a finite direct sum of isoclinic quasi-special $F$-crystals $(M_i, \varphi_i)$. We first prove the theorem for each $(M_i, \varphi_i)$. For $i=1,2,\dots,t$, let $r_i$ be the rank of $M_i$, $s_i$ the sum of all Hodge slopes of $(M_i, \varphi_i)$, and $e^{(i)}_{r_i}$ the largest Hodge slope of $(M_i, \varphi_i)$. By Lemma \ref{lemma:computelquasispecial}, we have
\[n_{M_i} = \max\{\delta_{M_i}(1), \delta_{M_i}(2), \dots, \delta_{M_i}(r_i)\}.\]
For each $1 \leq j \leq r_i$, we have $\alpha_{M_i}(j) \geq 0$ and $\beta_{M_i}(j) \leq s_i$. Therefore $\delta_{M_i}(j) \leq s_i$ and thus $n_{M_i} \leq s_i$. To show that $n_{M_i} \leq r_ie^{(i)}_{r_i}-s_i$, we consider the $F$-crystal $(M_i^*, p^{e^{(i)}_{r_i}}\varphi_i)$. It is an isoclinic quasi-special $F$-crystal whose isomorphism number is equal to the isomorphism number of $(M_i, \varphi_i)$ by Remark \ref{remark:doubleduality}. The sum of all Hodge slopes of $(M_i^*, p^{e^{(i)}_{r_i}}\varphi_i)$ is equal to $r_ie^{(i)}_{r_i}-s_i$. By using the same type of argument as before, we get that $n_{M_i} = n_{M^*_i} \leq r_ie^{(i)}_{r_i}-s_i$. Therefore, we have proved the theorem for each isoclinic quasi-special $F$-crystal $(M_i, \varphi_i)$, namely  
\begin{equation} \label{equation:quasispecial}
n_{M_i} \leq \min\{s_i, r_ie^{(i)}_{r_i}-s_i\}.
\end{equation}

Now we prove the theorem for $(M, \varphi)$. By Proposition \ref{proposition:estimateldirectsum}, we have $n_M \leq \max\{1, n_{M_i}, n_{M_i} + n_{M_j} - 1 \; | \; i, j \in I, i \neq j\} \leq \max\{1, n_{M_i} + n_{M_j} \; | \; i, j \in I, i \neq j\}.$ By \eqref{equation:quasispecial}, we have 
\[n_M \leq \max\{1, \min\{s_i+s_j, r_ie^{(i)}_{r_i}+r_je^{(j)}_{r_j}-s_i-s_j\}\; | \; i, j \in I, i \neq j\}.\] 
As $\sum_{l=1}^r s_l = s$, we have $s_i+s_j \leq s$. For any $1 \leq l \leq t$, as $e_{r_l}^{(l)} \leq e_r$ and $s_l \leq r_l e_{r_l}^{(l)} \leq r_le_r$, we have
\[r_ie^{(i)}_{r_i}+r_je^{(j)}_{r_j} + \sum_{l \neq i, j} s_l \leq r_i e_r + r_j e_r + \sum_{l \neq i, j} r_l e_r  = \sum_{l=1}^t r_l e_r = re_r.\]
Use this estimate, we get
\[r_ie^{(i)}_{r_i}+r_je^{(j)}_{r_j}-s_i-s_j  = (r_ie^{(i)}_{r_i}+r_je^{(j)}_{r_j} + \sum_{l \neq i, j} s_l) - \sum_{l=1}^t s_l \leq re_r - s.\] Thus $n_M \leq \max\{1, \min\{s, re_r-s\}\}$. If $\min\{s, re_r-s\} = 0$, then either $s=0$ or $re_r=s$. In both cases, the Hodge polygon of $(M, \varphi)$ is a straight line. By Proposition \ref{proposition:n=0}, we know that $n_M = 0$. Therefore $n_M \leq \min\{s, re_r-s\}$ as desired.
\end{proof}

\begin{example} \label{example:quasispecialoptimal}
Let $(M, \varphi)$ be a quasi-special $F$-crystal such that $s=e_r$. We claim that $n_M = \min\{s, re_r-s\}$. Indeed, if $r=1$, then $(M, \varphi)$ is an isoclinic ordinary $F$-crystal. In this case, the isomorphism number $n_M = 0 = \min\{s,re_r-s\}$. If $r > 1$, then $\min\{s,re_r-s\} = e_r$. By Lemma \ref{lemma:computelquasispecial}, we know that $n_M \geq \delta_M(1) = e_r$. Therefore $n_M = e_r$.
\end{example}

\begin{remark} \label{remark:quasispecial} \mbox{}
\begin{enumerate}
\item If $(M, \varphi)$ is a quasi-special Dieudonn\'e module with dimension $d$ and codimension $c$, then $e_r = 1$ and $s = d$. By Theorem \ref{theorem:quasi-special}, we have $n_M \leq \min \{c, d\}$. This recovers \cite[Theorem 1.5.2]{Vasiu:reconstructing}.
\item Theorem \ref{theorem:quasi-special} is not optimal in general. For example, if $(M, \varphi)$ is a quasi-special $F$-crystal of K3 type, then by Theorem \ref{theorem:K3}, $n_M \leq 2$. On the other hand, Theorem \ref{theorem:quasi-special} asserts that $n_M \leq r$.
\end{enumerate}
\end{remark}

Let $\{v_1, v_2, \dots, v_r\}$ be a $W$-basis of $M$. Let $\pi$ be an arbitrary permutation of the set $\{1, 2, \dots, r\}$. Let ${\bf e} := \{e_1 \leq e_2 \leq \dots \leq e_r\}$ be a sequence of non-negative integers. The $F$-crystal $(M, \varphi_{\pi, {\bf e}})$ is defined by the rule $\varphi_{\pi, {\bf e}}(v_i) = p^{e_i}v_{\pi(i)}$ for all $1 \leq i \leq r$. Clearly the Hodge slopes of $(M, \varphi_{\pi, {\bf e}})$ are $e_1, e_2, \dots, e_r$.

\begin{definition} \label{definition:permutationalcyclic}
An $F$-crystal $(M, \varphi)$ is called \emph{permutational} (resp. \emph{cyclic}) if there is a non-trivial permutation (resp. cycle) $\pi$ such that $(M, \varphi)$ is isomorphic to $(M, \varphi_{\pi,{\bf e} })$ where ${\bf e} := \{e_1 \leq e_2 \leq \dots \leq e_r\}$ are the Hodge slopes of $(M, \varphi)$.
\end{definition}

\begin{remark} \mbox{} \label{remark:permuisquasi}
\begin{enumerate}
\item If $(M, \varphi)$ is permutational, then $(M, \varphi)$ is quasi-special. See \cite[Lemma 4.2.4(a)]{Vasiu:reconstructing} for a proof of the same result for $p$-divisible groups.
\item If $(M, \varphi)$ is cyclic of rank $r$, then $\varphi^r(M) = p^sM$ where $s$ is the sum of all Hodge slopes. Hence $(M, \varphi)$ is isoclinic with unique Newton slope equal to $s/r$.
\end{enumerate}
\end{remark}

We turn our attention to the isomorphism number of permutational $F$-crystals. By Propositions \ref{proposition:computelii} and \ref{proposition:computelisoclinic}, given an explicit formula of $\varphi$ in terms of a permutation, it is not hard to compute $n_M$ of a permutational $F$-crystal. In the next proposition, we study the maximal possible value of $n_M$ if we only know the Hodge slopes of $(M, \varphi)$ without knowing an explicit formula of $\varphi$.

\begin{lemma} \label{lemma:permutational}
Let $e_1 \leq e_2 \leq \cdots \leq e_r$ be integers. Fix $j \in \{1, 2, \dots, \lfloor r/2 \rfloor\}$. For any $s_1, s_2, \dots, s_j, t_1, t_2, \dots, t_j \in \{1, 2, \dots, r\}$ such that
\begin{enumerate}[(a)]
\item $s_1, s_2, \dots, s_j$ are distinct and $t_1, t_2, \dots, t_j$ are distinct;
\item $e_{t_1} \leq e_{t_2} \leq \cdots \leq e_{t_j}$ and $e_{s_j} \leq e_{s_{i-1}} \leq \cdots \leq e_{s_1}$;
\item $\alpha := e_{t_1} + e_{t_2} + \cdots + e_{t_j} \leq e_{s_1} + e_{s_2} + \cdots + e_{s_j} =: \beta$; 
\end{enumerate}
we have $\beta - \alpha \leq \sum_{i=1}^j (e_{r-i+1} - e_i)$.
\end{lemma}
\begin{proof}
As $e_{t_1} \leq \alpha \leq \beta \leq e_{s_1}$, we can define $l \in \{1, 2, \dots, j\}$ to be the largest number such that $e_{t_l} \leq e_{s_l}$. Therefore, we have 
\[e_{t_1} \leq e_{t_2} \leq \cdots e_{t_l} \leq e_{s_l} \leq \cdots \leq e_{s_2} \leq e_{s_1}.\]
It is easy to see that $e_{s_i} - e_{t_i} \leq e_{r-i+1} - e_i$ for all $1 \leq i \leq l$. If $l<j$, we have $e_{s_i} - e_{t_i} < 0 \leq e_{r-i+1} - e_i$ for all $l < i \leq j$. To conclude the proof, we just have to sum up the inequalities $e_{s_i} - e_{t_i} \leq e_{r-i+1} - e_i$ for all $1 \leq i \leq j$.
\end{proof}

\begin{proposition} \label{proposition:permutational}
Let $(M, \varphi)$ be a permutational $F$-crystal with Hodge slopes ${\bf e} = \{e_1 \leq e_2 \leq \cdots \leq e_r\}$. Then the following inequality holds
\begin{equation} \label{equation:inumpermutational}
n_M \leq \sum_{i=1}^{\lfloor r/2 \rfloor} (e_{r-i+1} - e_i).
\end{equation}
Furthermore, the inequality is optimal in the sense that for every choice of Hodge slopes $e_1, e_2, \dots, e_r$, there is a permutational $F$-crystal such that \eqref{equation:inumpermutational} is an equality.
\end{proposition}
\begin{proof}
We first prove the inequality for cyclic $F$-crystals. Let $\pi$ be a cycle such that $(M, \varphi) \cong (M, \varphi_{\pi, {\bf e}})$. Since every cyclic $F$-crystal is an isoclinic quasi-special $F$-crystal by the second part of Remark \ref{remark:permuisquasi}, the isomorphism number $n_M$ of $(M, \varphi)$ is $\max\{\delta_M(j) \; | \; j = 1, 2, \dots, r\}$ by Lemma \ref{lemma:computelquasispecial}. For each $j \in \{1, 2, \dots, r\}$,  the Hodge slopes of $(M, \varphi^j)$ are 
\[\sum_{i=0}^{j-1} e_{\pi^i(1)}, \quad \sum_{i=0}^{j-1} e_{\pi^i(2)}, \quad \dots, \quad \sum_{i=0}^{j-1} e_{\pi^i(r)}.\]
Then $\delta_M(j)$ is the difference between the maximum number, that is $\beta_M(j)$,  and the minimum number, that is $\alpha_M(j)$, from the above list. For each $j \in \{1, 2, \dots, r\}$, we claim that $\delta_M(j) = \delta_M(r-j)$. Indeed, this can be easily checked by observing that the Hodge slopes of $(M, \varphi^{r-j})$ are
\[s- \sum_{i=0}^{j-1} e_{\pi^i(1)}, \quad s- \sum_{i=0}^{j-1} e_{\pi^i(2)}, \quad \dots, \quad s- \sum_{i=0}^{j-1} e_{\pi^i(r)},\]
with $s = \sum_{i=1}^r e_i$. Therefore $n_M = \max\{\delta_M(j) \; | \; j = 1, 2, \dots, \lfloor r/2 \rfloor\}$.

Applying Lemma \ref{lemma:permutational} to $\beta = \beta_M(j)$ and $\alpha = \alpha_M(j)$, we have for all $j \in \{1, 2, \dots, \lfloor r/2 \rfloor\}$, 
\[\delta_M(j) \leq \sum_{i=1}^{\lfloor r/2 \rfloor} (e_{r-i+1} - e_i).\]
This proves the proposition for cyclic $F$-crystals. Let $\pi = (1, 2, \dots, r)$, so $n_M = \delta_M(\lfloor r/2 \rfloor ) =  \sum_{i=1}^{\lfloor r/2 \rfloor} (e_{r-i+1} - e_i)$. This shows that the inequality \eqref{equation:inumpermutational} can be an equality for any choice of Hodge slopes in the cyclic $F$-crystal case.

If  $(M, \varphi) \cong (M, \varphi_{\pi, {\bf e}})$ is a permutational $F$-crystal for some non-trivial permutation $\pi$, then $(M, \varphi)$ is a finite direct sum of (possibly) two or more cyclic $F$-crystals, say $(M, \varphi) \cong \bigoplus_{i \in I} (M_i ,\varphi_i)$. As $\pi$ is non-trivial, we know that $n_{M_i} \geq 1$ for some $i$.

Applying the (proved) conclusion of Proposition \ref{proposition:permutational} to the cyclic $F$-crystals $(M_i, \varphi_i)$, we deduce that
\[n_{M_i} \leq \sum_{l=1}^{\lfloor r_i/2 \rfloor} (e^{(i)}_{r_i-l+1} - e^{(i)}_l)\]
where $e^{(i)}_1 \leq e^{(i)}_2 \leq \dots \leq e^{(i)}_{r_i}$ are the Hodge slopes of $(M_i, \varphi_i)$ and $r_i$ are the rank of $M_i$ for all $i \in I$. Applying Lemma \ref{lemma:permutational} to $\beta = \sum_{l=1}^{\lfloor r_i/2 \rfloor} e_{r_i-l+1}^{(i)}$ and $\alpha = \sum_{l=1}^{\lfloor r_i/2 \rfloor} e_l^{(i)}$, we have \[n_{M_i} \leq \beta - \alpha \leq \sum_{l=1}^{\lfloor r_i/2 \rfloor}(e_{r-l+1}-e_l) \leq \sum_{l=1}^{\lfloor r/2 \rfloor}(e_{r-l+1} - e_l).\]
Proposition \ref{proposition:estimateldirectsum} implies that $n_M \leq \max\{1, n_{M_i}, n_{M_i}+n_{M_j}-1 \; | \; i,j \in I, i \neq j \},$ so to prove Proposition \ref{proposition:permutational} in general, it suffices to show that
\begin{equation} \label{inequality:lasttoprove}
n_{M_i} + n_{M_j} - 1 < \sum_{l=1}^{\lfloor r/2 \rfloor}(e_{r-l+1} - e_l).
\end{equation}
For $i, j \in I$, $i \neq j$, we compute that
\begin{equation} \label{equation:inequalitypermutational}
n_{M_i} + n_{M_j} - 1 < (\sum_{l=1}^{\lfloor r_i/2 \rfloor} e^{(i)}_{r_i-l+1} +\sum_{l=1}^{\lfloor r_j/2 \rfloor} e^{(j)}_{r_j-l+1}) - (\sum_{l=1}^{\lfloor r_i/2 \rfloor} e^{(i)}_{l} +\sum_{l=1}^{\lfloor r_j/2 \rfloor} e^{(j)}_{l}).
\end{equation}
By Lemma \ref{lemma:permutational}, letting
\[\beta = \sum_{l=1}^{\lfloor r_i/2 \rfloor} e^{(i)}_{r_i-l+1} +\sum_{l=1}^{\lfloor r_j/2 \rfloor} e^{(j)}_{r_j-l+1} \quad \textrm{and} \quad \alpha = \sum_{l=1}^{\lfloor r_i/2 \rfloor} e^{(i)}_{l} +\sum_{l=1}^{\lfloor r_j/2 \rfloor} e^{(j)}_{l},\]
we have
\begin{equation} \label{inequality:lastlastprove}
\beta - \alpha \leq \sum_{l=1}^{\lfloor r_i/2 \rfloor + \lfloor r_j/2 \rfloor} (e_{s_l} - e_{t_l}) \leq \sum_{l=1}^{\lfloor r/2 \rfloor}(e_{r-l+1} - e_l).
\end{equation}
The last inequality is true because $\lfloor r_i/2 \rfloor + \lfloor r_j/2 \rfloor \leq \lfloor r/2 \rfloor$. Now \eqref{inequality:lasttoprove} is clear by \eqref{equation:inequalitypermutational} and \eqref{inequality:lastlastprove}, which completes the proof of Proposition \ref{proposition:permutational}.
\end{proof}

\begin{remark}
If $(M, \varphi)$ is a direct sum of two or more cyclic $F$-crystals, then $n_M < \sum_{l=1}^{\lfloor r/2 \rfloor} (e_{r-l+1}-e_l)$.
\end{remark}

\section*{References}
\bibliography{reference}

\end{document}